\documentclass[12 pt,a4paper]{amsart} 
\usepackage{amsfonts,amssymb,amscd,amsmath,enumerate,verbatim,calc} 
\usepackage{float}
\usepackage{amsthm}

\newtheorem{theorem}{Theorem}[section]
\newtheorem{definition}{Definition}[section]

\newtheorem{lemma}[theorem]{Lemma}
\newtheorem{corollary}[theorem]{Corollary}
\newtheorem{proposition}[theorem]{Proposition}
\newtheorem{remark}[theorem]{Remark}

\usepackage{mathtools}

\numberwithin{equation}{section}

\usepackage{amsfonts}

\begin{document}
	
		\title{Lie nilpotency and Lie solvability of tensor product of multiplicative Lie algebras}
			\author{Deepak Pal$^{1}$, Amit Kumar$^{2}$, Sumit Kumar Upadhyay$^{3}$ and Seema Kushwaha$^{4}$\vspace{.4cm}\\
		{Department of Applied Sciences,\\ Indian Institute of Information Technology Allahabad\\Prayagraj, U. P., India} }
	
	\thanks{$^1$deepakpal5797@gmail.com, $^2$amitiiit007@gmail.com, $^3$upadhyaysumit365@gmail.com, $^4$seema28k@gmail.com}
	\begin{abstract} In this article, we discuss Lie nilpotency  and Lie solvability of non-abelian tensor product of multiplicative Lie algebras. In particular, for giving information concerning the Lie nilpotency (or Lie solvability) of either multiplicative Lie algebras $G$ or $H,$ the non-abelian tensor product $ \frac{G\otimes H}{I}$  is Lie nilpotent (or Lie solvable), for some ideal $I$ of ${G\otimes H}$.
		
	\end{abstract}
	\maketitle 
	
	\section{Introduction}
 	The non-abelian tensor square of a group was introduced by R.K. Dennis \cite{RK} in a search for ``homology functors having a close relationship to K-theory". In \cite{RJ}, R. Brown and J.-L. Loday  gave a topological significance for the tensor square and defined the tensor product of two distinct groups acting ``compatibly" on each other. In \cite{RDE} R. Brown, D.L. Johnson and E.F. Robertson focus on the group theoretic properties of the non-abelian tensor square. In particular, they discussed the nilpotency and solvability of non-abelian tensor squares of groups. The concepts of nilpotency class and solvability length provide a systematic way to categorize groups based on their structural properties.  In 1999, M. P. Visscher \cite{MP} provided bounds on the nilpotency class and solvability length of  non-abelian tensor product.
 
 G.J. Ellis \cite{GJ1} introduced the notion of the non-abelian  tensor product for a pair of Lie algebras and investigated some of its fundamental properties. Furthermore, in \cite{AHHB}, Salemkar et al. studied some common properties between Lie algebras and their tensor products, and presented some bounds on the nilpotency class and solvability length of  non-abelian  tensor product.
 
 In 1993, Ellis \cite{GJ} introduced the notion of multiplicative Lie algebra which generalizes the concept of groups and Lie rings. It has been interesting to generalize the results on groups and Lie rings to multiplicative Lie algebras. In \cite{GNM}, Donadze et al. introduced the concept of non-abelian tensor product of multiplicative Lie algebras. They proved that this notion recovers the existing  notions of non-abelian tensor products of groups as well as of Lie algebras.  A natural question arises regarding the properties of the non-abelian tensor product $G \otimes H$ for two multiplicative Lie algebras $G$ and $H$. Specifically, one may inquire about the implications for nilpotency and solvability of $G \otimes H$ based on the corresponding properties of $G$ and $H$. In their study \cite{GLW}, Donadze et al. explored the nilpotency and solvability of the non-abelian tensor product of multiplicative Lie algebras, employing the concept of nilpotency introduced by Point and Wantiez in \cite{FP}.
 On a group $G$, there is two multiplicative Lie algebra structures viz. the trivial one given by $x\star y = 1$ and the improper given by $x\star y = [x,y]$ for all $x,y \in G$. It is of interest to get smallest ideal $I$ of $G$ such that the induced structure on $\frac{G}{I}$ is improper. To do so,  in 2021,  Pandey and Upadhyay \cite{MRS} developed a different notion of solvability and nilpotency of multiplicative Lie algebras with the help of multiplicative Lie center and Lie commutator and named it as Lie solvability  and Lie nilpotency.
 
 In this article, we go with the notion of solvability and nilpotency given in \cite{MRS} and discuss about Lie nilpotency and Lie solvability of the non-abelian tensor product of multiplicative Lie algebras . As our main results, we provide upper bounds on the Lie nilpotency class and Lie solvability length of $\frac{G\otimes H}{I},$ for some ideal $I$ of $G\otimes H,$ in terms of the Lie nilpotency class and Lie solvability length of $H$.      
 	
	\section{Preliminaries} 
We start the section with the definition of multiplicative Lie algebra, and some basics of Lie solvability  and Lie nilpotency. After that we recall the tensor product of groups as well as multiplicative Lie algebras. 
 	
\begin{definition}
	A multiplicative Lie algebra is a triple $ (G,\cdot,\star), $ where $G$ is a set, $ \cdot $ and $ \star $ are two binary operations on $G$ such that $ (G,\cdot) $ is a group (need not be abelian) and for all  $x, y, z  \in  G$, the following identities hold:
\end{definition} 
	\begin{enumerate}
		\item $ x\star x=1 $ , 
		\item $ x\star(y \cdot z)=(x\star y)\cdot{^y(x\star z)} $, 
		\item $ (x \cdot y)\star z= {^x(y\star z)} \cdot (x\star z) $ ,
		\item $ ((x\star y)\star^yz)((y\star z)\star^zx)((z\star x)\star ^xy)=1 $ ,
		\item $ ^z(x\star y)=(^zx\star ^zy)$ 
	\end{enumerate} where $^zx$ denotes the element $zxz^{-1}$.
	We  say that $ \star $ is a multiplicative Lie algebra structure on the group $G$.

\begin{definition} \cite{MRS} 
	Let $G$ be a multiplicative Lie algebra. The ideal $^L[G, G]$ generated by the set $\{(a\star b)^{-1}[a, b]\mid a, b \in G\}$ of $G$ is called  the Lie commutator of $G.$ The element $(a\star b)^{-1}[a, b]$ is denoted by $^L[a, b]$.
\end{definition}

\begin{proposition}\cite{MRS} \label{Lie}
	Let $G$ be a multiplicative Lie algebra.  Then following identities holds:
\end{proposition}
\begin{enumerate}
	\item  $^L[a, a] = 1, \ \forall \ a \in G,$ 
	
	\item  $^L[a, b] \ ^L[b, a] = 1, \ \forall \ a, b \in G,$ \vspace{0.12cm}
	
	\item $^L[ab, c] = {^L[a, c]} \ ^{(^c{a})} {(^L[b, c]),} \ \forall \ a, b, c \in G, $ 
	
	\item $ ^L[a, bc] = {^{b}{(^L[a, c])}} \  {^{[^{b}c, ^{b}a]}({^L[a, b]})}, \ \forall \ a, b, c \in G,$ 
	
	\item $ ^a{(^L[b, c])} = {^L[^ab, ^ac]}, \ \forall \ a, b, c \in G, $ \vspace{0.12cm}
	
	\item $^L[a^{-1}, b] = {^{a^{-1}}{(^L[b, a])}} \ \text{and} \ ^L[a, b^{-1}] = {^{b^{-1}}{(^L[b, a])}}, \ \forall \ a, b \in G,$ 
	
	\item $^{^L[a, b]}{(x \star y)} = {(x \star y)}. $ In particular, $[^L[a, b], {(x \star y)}] = 1,  \ \forall \ a, b, x, y \in G.$
\end{enumerate}

\begin{definition}\cite{MRS} Let $G$ be a multiplicative Lie algebra. Then
\end{definition}
 \begin{enumerate}
		\item $G$ is said to be Lie solvable if $G^{(n)} = \{1\}$
		for some $n$, where $G^{(0)} = G,$ and $G^{(k+1)} = {^L[G^{(k)}, G^{(k)}]}$ for all $k\in \mathbb{N}$.
		
		\item $G$ is said to be Lie nilpotent if $L_n(G) = {1}$
		for some $n$, where $G =L_0(G)$, and $L_{k+1}(G) =  {^L[G, L_k(G)]}$ for all $k\in \mathbb{N}$.
		
	\end{enumerate}

\subsection{Tensor product of groups}		
We say that groups $G$ and $H$ act  compatibly on each other if
\begin{equation} \label{C1}
^{(^gh)}g' = {^{ghg^{-1}}g'} ~\text{and}~   ^{(^hg)}h' = {^{hgh^{-1}}h'},
\end{equation}
where $^hg$ and $^gh$ denote the action of $h$ on $g$ and the action of $g$ on $h$, respectively.
\begin{definition}
	If $G$ and $H$ are groups which act compatibly on each other, then the  non-abelian tensor product $G \otimes H$ is the group generated by the symbols  $x \otimes y$ (for all $x \in G $ and $y \in H$) subject to the following relations:
\end{definition}
	\begin{enumerate}
		\item 	$x \otimes (yy') = (x \otimes y) (^yx \otimes ^yy')$, 
		\item   $(xx') \otimes y = (^xx' \otimes ^xy)(x \otimes y)$ .
	\end{enumerate}

\begin{proposition}\cite{RJ} \label{s2}
	Let $G$ and $H$ be groups acting on each other compatibly. Then $G$ and $H$ act on $G\otimes H$ by
	
	\hspace{2.5cm}	${^g(g' \otimes h )} = ({^g{g'}}\otimes {^gh}), \ \  ^h(g \otimes h' ) = (^h{g}\otimes {^h{h'}}),$
	
	for all $g,g' \in G \ \text{and} \ h,h' \in H. $
\end{proposition}

\begin{proposition}\cite{RJ} \label{s1}
	The following identities hold in $G\otimes H$ for all $g,g' \in G \ \text{and} \ h,h' \in H $
\end{proposition}
\begin{enumerate}
	\item $1_G \otimes h = 1_{G\otimes H} = g\otimes 1_H$ 
	
	\item $(g \otimes h)^{-1} = {^g(g^{-1} \otimes h)} = {^h(g \otimes h^{-1})}$ 
	\item $(g \otimes h)(g' \otimes h')(g \otimes h)^{-1} = {^{[g, h]}(g' \otimes h')} $ 
	
	\item $(g {^h{g^{-1}}}) \otimes h' = (g \otimes h)^{h'} (g \otimes h)^{-1}$ 
	\item $ g' \otimes ({^gh}h^{-1}) = {^{g'}(g \otimes h) (g \otimes h)^{-1}}$ 
	
	\item $[g \otimes h, g' \otimes h'] = (g^h{g^{-1}}) \otimes (^{g'}{h'}{h'}^{-1}). $
	
\end{enumerate}

\subsection{Tensor product of multiplicative Lie algebras}

\begin{definition} \cite{GNM}
	Let $G$ and $H$ be two multiplicative Lie algebras. By an action of $G$ on $H$ we mean an underlying group action of  $G$ on  $H$, given by a group homomorphism $\phi: G \to Aut(H)$ ( $H$ also acts on $G$), together with a map $G \times H \to H, (x, y) \to \langle x, y \rangle,$ satisfying the
	following conditions:
\end{definition}
\hspace{2.5cm}	$ \langle x, yy' \rangle  = \langle x, y\rangle \langle ^yx, {^yy'}\rangle, $

\hspace{2.5cm}   $\langle xx', y \rangle  = \langle ^xx', {^xy} \rangle \langle x, y\rangle,$ 

\hspace{2.5cm}	$\langle (x \star x'), {^{x'}y} \rangle \langle ^yx, \langle x', y\rangle \rangle ^{-1} \langle ^xx', {\langle x, y \rangle} ^{-1}\rangle^{-1} = 1, $

\hspace{2.5cm}	$\langle ^{y'}x, (y\star y')\rangle \langle  {\langle y, x\rangle}^{-1}, {^yy'}\rangle^{-1} \langle  {\langle y', x\rangle}, {^xy}\rangle ^{-1} = 1,$  

where $x, x' \in G, y, y'\in H, {^xy}= \phi(x)(y), {^xx'} = xx'x^{-1}, {^yy'} = yy'y^{-1}.$

\begin{definition} \cite{GNM} \label{s3}
	Let $G$ and $H$ be two multiplicative Lie algebras acting on each other. Then the non-abelian tensor product $G \otimes H$ is the multiplicative Lie algebra generated by the symbols $x \otimes y$ (for all $x \in G $ and $y \in H$) subject to the following relations:
\end{definition}
\begin{enumerate}
	\item 	$x \otimes (yy') = (x \otimes y) {^y(x \otimes y')}$ \label{C11} 
	
	\item   $(xx') \otimes y = {^x(x' \otimes y)} (x \otimes y)$ \label{C7}
	
	\item $((x \star x') \otimes {^{x'}y}) (^yx \otimes \langle x', y\rangle)^{-1} ({^xx'} \otimes {\langle x, y \rangle} ^{-1})^{-1} = 1$ \label{C8}
	
	\item $({^{y'}x}  \otimes (y\star y')) ( {\langle y, x\rangle}^{-1}  \otimes {^yy'})^{-1} ( {\langle y', x\rangle}  \otimes {^xy})^{-1} = 1$ \label{C9} 
	
	\item $ ((x \otimes y) \star (x' \otimes y')) = {\langle y, x \rangle}^{-1} \otimes {\langle x', y' \rangle}. $ \label{C10}
	
\end{enumerate}

for all $g, g' \in G \ \text{and} \ h, h' \in H.$

\begin{proposition} \cite{GLW} \label{p1} 
	Let $G$ and $H$ be two multiplicative Lie algebras acting on each other. If the actions of the underlying groups are compatible, then 	for all $x, g \in G \ \text{and} \ y, h \in H$\\
	
	\ \ \ 	$ \langle x, y \rangle \langle g, h \rangle {\langle x, y \rangle}^{-1} = \langle {^{[x, y]}g}, {^{[x, y]}h} \rangle,  \ \ \  \langle y, x \rangle \langle h, g \rangle {\langle y, x \rangle}^{-1} = \langle {^{[y, x]}h}, {^{[y, x]}g} \rangle, $ \vspace{0.2cm}\\
where $[x, y]$ and $[y, x]$ denote the $ ^x{y{y^{-1}}}$ and $^yx{{x^{-1}}},$ respectively. 

\end{proposition}
	\section{On the nilpotency and solvability of tensor products}
In this section we establish interrelationship between the ideals ${^L}[H,G]$ and ${^L}[G,H]$ of a pair of multiplicative Lie algebras $G$ and $H$ and give upper bounds on the nilpotency class and solvability length of $ \frac{G \otimes H}{I},$ for some ideal $I$ of $G \otimes H,$ in terms of the nilpotency class and solvability length of $H.$ Now, we start the section by introducing few more compatible conditions of multiplicative Lie algebras $G$ and $H$ which are useful for further development.
\begin{definition}\label{1}
Let $G$ and $H$ be multiplicative Lie algebras acting on each other. The actions are said to be compatible if 
\end{definition}
\begin{enumerate}
\item $^{(^gh)}g' = {^{ghg^{-1}}g'} ~\text{and}~   ^{(^hg)}h' = {^{hgh^{-1}}h'}$;
	\item $\langle \langle h,g \rangle^{-1},h' \rangle = \langle g,h \rangle \star h'$ and  $\langle \langle g,h \rangle^{-1},g' \rangle = \langle h,g \rangle \star g'$;\label{C2} \vspace{0.12cm}
	
	\item $^{\langle g,h \rangle \langle h,g \rangle}h' = h'$ and $^{\langle g,h \rangle \langle h,g \rangle}g' = g' $;\label{C3} \vspace{0.12cm}
	
	\item $^g{\langle h,g' \rangle} = \langle {^gh},{^gg'} \rangle$ and $^h{\langle g,h' \rangle} = \langle {^hg},{^hh'} \rangle $;\label{C4} \vspace{0.12cm}
	
	\item $ \langle g {^h{g^{-1}}} ,h' \rangle = ({^gh}h^{-1}) \star h' $ and $ \langle h {^g{h^{-1}}} ,g' \rangle = ({^hg}g^{-1}) \star g' $,\label{C5}
\end{enumerate}
for all $g, g' \in G \ \text{and} \ h, h' \in H.	$ 

\begin{remark}
The compatible actions of  multiplicative Lie algebras are already defined in \cite{GLW}. Here, we have added two more conditions.	
\end{remark}
	
	\begin{lemma}\label{C6}
		Let $G$ and $H$ be multiplicative Lie algebras acting on each other compatibly. Then for all $g\in G$ and $h,h'\in H,$ \vspace{0.2cm}
		
	\hspace{2cm}	$	[\langle g,h \rangle,h'] =  \langle g {^h{g^{-1}}} ,h' \rangle = ({^gh}h^{-1}) \star h'.$
	\end{lemma}
	\begin{proof} To prove our lemma, it suffices to show that 	$	[\langle g,h \rangle,h'] =  \langle g {^h{g^{-1}}} ,h' \rangle, $ for  $g\in G$ and $h,h'\in H.$ Now 
		\begin{align*}
			[\langle g,h \rangle,h'] = \langle g,h \rangle {(^{h'}\langle g,h \rangle^{-1})} 
			&= \langle g,h \rangle {({^{h'}{(^g\langle g^{-1}},h \rangle)})} \\
			&= \langle g,h \rangle \langle g,h' \rangle^{-1} \langle g,h' \rangle {({^{h'}{(^g\langle g^{-1}},h \rangle)})}  \langle g,h' \rangle^{-1} \langle g,h' \rangle \\
			&=  \langle g,h \rangle \langle g,h' \rangle^{-1} {({^{^g{h'}}{(^g\langle g^{-1}},h \rangle)})}  \langle g,h' \rangle \\
			&= \langle g,h \rangle \langle g,h' \rangle^{-1} {({{^g}{(^{h'} \langle g^{-1}},h \rangle)})}  \langle g,h' \rangle \\
			&=  \langle g,h \rangle \langle gg^{-1},h' \rangle \langle g,h' \rangle^{-1} {({{^g}{(^{h'} \langle g^{-1}},h \rangle)})}  \langle g,h' \rangle \\
			&=  \langle g,h \rangle {(^g{\langle g^{-1},h' \rangle})}  {({{^g}{(^{h'} \langle g^{-1}},h \rangle)})}  \langle g,h' \rangle \\
			&=  \langle g,h \rangle {(^g{\langle g^{-1},h'h \rangle})}  \langle g,h' \rangle \\
			&=  \langle g,h \rangle \langle gg^{-1},h \rangle^{-1}  {(^g{\langle g^{-1},h'h \rangle})}  \langle g,h' \rangle \\
			&=   {^g{\langle g^{-1},h \rangle^{-1}}}  {(^g{\langle g^{-1},h'h \rangle})}  \langle g,h' \rangle \\
			&=   {^g{\langle g^{-1},h \rangle^{-1}}}  {(^g{\langle g^{-1},hh^{-1}h'h \rangle})}  \langle g,h' \rangle \\
			&=   {^g{{(^h{\langle g^{-1},h^{-1}h'h \rangle})}}}  \langle g,h' \rangle \\
			&=   {^g{{\langle ^hg^{-1},h' \rangle}}}  \langle g,h' \rangle \\
			&=   \langle g {^h{g^{-1}}} ,h' \rangle. 
		\end{align*}
		Now, by compatibility condition, the result holds.
	\end{proof}
		
	\begin{proposition}\label{Ideal}
		Let $G$ and $H$ be multiplicative Lie algebras acting on each other compatibly. Then
			\end{proposition}
		\begin{enumerate}
			\item The subgroup ${^GH}$ generated by the set $\{{^gh}h^{-1}:g\in G, h\in H\}$ is an ideal of $H.$
			
			\item The subgroup $ \langle G,H \rangle $ generated by the set $\{\langle g,h \rangle:g\in G, h\in H\}$ is an ideal of $H.$
			
			\item The subgroup $^L[G,H]$ generated by the set $\{\langle g,h \rangle^{-1}({^gh}h^{-1}):g\in G, h\in H\}$ is an ideal of $H$ and element $\langle g,h \rangle^{-1}({^gh}h^{-1})$ of $^L[G,H]$ is denoted by  $^L[g,h].$  
	\end{enumerate} 

	\begin{proof}$(1)$ Let $h' \in H$ and ${^gh}h^{-1} \in {^GH}. $ Then by group compatible condition, it is easy to see that ${^GH}$ is a normal subgroup of $H.$
			 
	Now by using Lemma \ref{C6}, and taking into account the compatibility conditions of actions, we have $	h' \star ({^gh}h^{-1}) = (({^gh}h^{-1}) \star h')^{-1} = [\langle g,h\rangle, h']^{-1} = {(^{\langle g,h\rangle}h'h'^{-1})^{-1}} = {(^{\langle h,g\rangle^{-1}}h'h'^{-1})^{-1}} \in {^GH}.$ 	Hence, ${^GH}$  is an ideal of $H.$ \vspace{0.15cm}\\
$(2)$ Let $h' \in H$ and $\langle g,h \rangle \in \langle G,H \rangle. $ Then by compatible conditions, it can be easily seen that $ \langle G,H \rangle $ is an ideal of $H.$ \vspace{0.15cm}\\
$(3)$ Suppose  $ g\in G \ \text{and} \ h, h' \in H,$ then by the compatibility conditions, we have  $ {^{h'}(\langle g,h \rangle^{-1}({^gh}h^{-1}))}  = {{^{h'}\langle g,h \rangle^{-1}}{({^{h'}({^gh}h^{-1})})}} = {\langle ^{h'}g,{^{h'}h} \rangle^{-1}}{(^{h'g}{(^{h'}h)}{^{h'}h^{-1}})}. $ Hence, any generator of $^L[G,H]$ is mapped to another generator under the action of conjugation by elements of $H.$ It follows $^L[G,H]$ is a normal subgroup of $H.$ Also
			\begin{align*}
			\hspace{-2.5cm}	h' \star (\langle g,h \rangle^{-1}({^gh}h^{-1})) 
				&= (h'\star \langle g,h \rangle^{-1}){(^{\langle g,h \rangle^{-1}}(h' \star ({^gh}h^{-1}) ))} \\ 
				&= {( \langle g,h \rangle^{-1} \star h')}^{-1}{({^{\langle g,h \rangle^{-1}}(h' \star ({^gh}h^{-1}))} )}\\
				&= {\langle  \langle h,g \rangle , h' \rangle}^{-1}{({^{\langle g,h \rangle^{-1}} {({^{h'}{\langle g,h \rangle}}{\langle g,h \rangle^{-1}}})})} \\
				&=  {\langle  \langle h,g \rangle , h' \rangle}^{-1} {({^{\langle h,g \rangle}}h'h'^{-1})}.
			\end{align*}
			Thus $^L[G,H]$  is an ideal of $H.$\\
	Similarly, we can also prove that ${^HG},$ $ \langle H, G \rangle $ and $^L[H,G]$ is an ideal of G. 
	\end{proof}

	\begin{proposition}\label{R1}
		Let $G$ and $H$ be multiplicative Lie algebras acting on each other in a  compatible way. Then
		\begin{align*}
			^{\langle g,h \rangle^{-1}({^gh}h^{-1})}g' = {^{\langle h,g \rangle {(g{^hg^{-1}})}}g'}  \ \ \text{and} \ \
			^{\langle g,h \rangle^{-1}({^gh}h^{-1})}h' = {^{\langle h,g \rangle {(g{^hg^{-1}})}}h'}
		\end{align*} 
		for all $g,g'\in G$ and $h,h'\in H.$
	\end{proposition}
	
	\begin{proof}
		Suppose	$g,g'\in G$ and $h\in H$, then by the compatibility conditions, we have
		\begin{align*}
			^{\langle g,h \rangle^{-1}({^gh}h^{-1})}g' = {^{\langle g,h \rangle^{-1}}({^{{^gh}h^{-1}}g'})} 
			= {^{\langle g,h \rangle^{-1}}(^{g{^hg^{-1}}}g')} 
			= {^{\langle h,g \rangle {(g{^hg^{-1}})}}g'}.
		\end{align*} 
		Similarly, the second equality holds.
	\end{proof}
	
	Now we extend the above proposition to all of $^L[H,G] $ and $^L[G,H] $.
\begin{proposition}\label{R2}
		Let $G$ and $H$ be multiplicative Lie algebras acting on each other in a  compatible way. Given $x\in {^L[G,H]} $, there exists $y\in {^L[H,G]} $ such that for all $g'\in G $ and $h'\in H$, 
		\begin{equation} \label{p}
			^xg' = {^yg'}  \ \ \text{and} \ \ ^xh' = {^yh'} .
		\end{equation}
	\end{proposition}	

	\begin{proof}
	Let $x\in ^L[G,H].$ By inducting on the word length $n$ of $x$ in terms of the generators
		of $^L[G,H] $ and their inverses, we will show there exists $y\in ^L[H,G] $ such that the pair $x, y$ satisfies (\ref{p}), that is $^xg' = {^yg'}$ and $ ^xh' = {^yh'}$ for all $g'\in G $ and $h'\in H.$ First suppose $n = 1.$ Then $x$ is either a generator, in which case the claim reduces to Proposition \ref{R1}, or $x = {({\langle g,h \rangle^{-1}({^gh}h^{-1})})}^{-1}$ for some $g\in G$ and $h\in H.$ In the latter case, setting $y = ({\langle h,g \rangle}{(g{^hg^{-1}})})^{-1} $, it follows from Proposition (\ref{p}) that  $^{y^{-1}}(^xg') = {^{x^{-1}}({^xg'})} = g' $ for $g'\in G.$ Thus, we have $^xg' = {^yg'} $ for all $g'\in G.$ Similarly, $ ^xh' = {^yh'} $ for all  $h'\in H.$ 
		
		Now suppose $n \geq 2.$ Then  $x = {({\langle g,h \rangle^{-1}({^gh}h^{-1})})}^{\epsilon}x'$, where $\epsilon = \pm 1$ and $x'$ has word length $n-1.$ By the induction hypothesis, there exists $y' \in {^L[H,G]}$ such that $^{x'}g' = {^{y'}g'}.$ Setting $y =  {({{\langle h,g \rangle {(g{^hg^{-1}})}})}^{\epsilon}y' \in {^L[H,G]}}  $, we get $^xg' = {^yg'} .$ Similarly, $ ^xh' = {^yh'} $ for all  $h'\in H.$
	\end{proof}
	
	\begin{remark} \label{d1}
	\begin{enumerate}
	\item If $G$ and $H$ be two multiplicative Lie algebras acting on each other compatibly, then by Proposition \ref{p1} and compatibility condition \ref{C4}, it becomes evident that the elements $ ^L[x, y]$
	and $ ^L[y, x],$ $(x \in G, y\in H)$ commute with the elements of $ \langle G, H \rangle$ and $ \langle H, G \rangle,$ respectively.
	\item By Proposition \ref{R2}, it is easy to see that the ideals $^L[G,H]$ and $^L[H,G]$ act trivially on the ideals $ \langle H,G \rangle $  and $ \langle G,H \rangle $, respectively.
	\end{enumerate}
	 
	\end{remark}
		
	Note that the previous proposition continues to hold if ${^L[G,H]} $ and ${^L[H,G]} $ are interchanged, that is given $y\in {^L[H,G]} $, there exists $x\in {^L[G,H]} $ such that the pair $x, y$ satisfies (3.1). Next we give a corollary of Proposition \ref{R2} which will be of use later. The
	proof is straightforward and thus is omitted.
	\begin{corollary}\label{R3}
		Let  $x_i\in {^L[G,H]} $, and $y_i\in {^L[H,G]}, i = 1,2, $ such that each pair $x_i, y_i$ satisfies $(3.1)$. Then the pairs ${x_i}^{-1}, {y_i}^{-1}$ and $[x_1,x_2], [y_1,y_2] $ also satisfy $(3.1)$.  
	\end{corollary}

	\begin{lemma}\label{R4}
		Let $G$ and $H$ be multiplicative Lie algebras acting on each other compatibly.
		Suppose $x_1, x_2\in {^L[G,H]} $, then there exists $y_1\in {^L[H,G]} $ such that  $x_1 \star x_2 = \langle y_1,x_2 \rangle.$
	\end{lemma}	
	\begin{proof}
		Without loss of generality, we may assume that  $x_1 = {\langle g,h \rangle^{-1}({^gh}h^{-1})}.$ Then by considering the compatibility conditions of actions, we have $x_1 \star x_2  = ({\langle g,h \rangle^{-1}({^gh}h^{-1})}) \star x_2 = {^{\langle g,h \rangle^{-1}}({({^gh}h^{-1}) \star x_2})}({\langle g,h \rangle^{-1}} \star \ x_2) = {^{\langle g,h \rangle^{-1}}\langle{g^h{g^{-1}} , x_2}}\rangle \langle{\langle h,g \rangle}, x_2\rangle =  {^{\langle h,g \rangle}\langle{g^h{g^{-1}}, x_2}}\rangle \linebreak \langle{\langle h,g \rangle} ,  x_2\rangle  = \langle{\langle h,g \rangle}({g^h{g^{-1}}}) , x_2\rangle = \langle y_1,x_2 \rangle ,  $ where $y_1 = {\langle h,g \rangle}({g^h{g^{-1}}})\in {^L[H,G]}.$ So the poof holds for the generator of ${^L[G,H]}.$ This then extends to general elements of ${^L[G,H]}$ by inducting on the word length of $n$ of $x_1$ in terms of the generators and their inverses. Hence, for any $x_1, x_2\in {^L[G,H]} $, there exists $y_1\in {^L[H,G]} $ such that  $x_1 \star x_2 = \langle y_1,x_2 \rangle.$   
	\end{proof}	

	\begin{lemma}\label{R6}
		Let $G$ and $H$ be multiplicative Lie algebras acting on each other compatibly.
		Suppose $x_1, x_2 \in {^L[G,H]} $, then there exists $y_1, y_2\in {^L[H,G]} $ such that for all $g\in G \ \text{and} \ h\in H$ 
		$$^{^L[x_1,x_2]}{g} = {^{^L[y_1,y_2]}{g}}, \  ^{^L[x_1,x_2]}{h} = {^{^L[y_1,y_2]}{h}}.$$
	\end{lemma}	
	\begin{proof}
		Let  $x_1, x_2 \in {^L[G,H]} $ and $g\in G.$ Then by using Corollary \ref{R3} and Lemma \ref{R4}, and considering the compatibility conditions of actions, we have $ 	^{^L[x_1,x_2]}{g} = {^{(x_1\star x_2)^{-1}}{(^{[x_1,x_2]}{g})}} = {^{\langle y_1, x_2 \rangle^{-1}}{(^{[y_1,y_2]}{g})}} =  {^{\langle y_1, x_2 \rangle^{-1}}{(^{[y_1,y_2]}{g})}} =  {^{\langle x_2, y_1 \rangle}{(^{[y_1,y_2]}{g})}} =  {^{(y_2 \star y_1)}{(^{[y_1,y_2]}{g})}} = {^{^L[y_1,y_2]}{g}}. $ Similarly, we can see 	$^{^L[x_1,x_2]}{h} = {^{^L[y_1,y_2]}{h}},$  for all $h \in H.$ 
		  
	\end{proof}	
	
	Now we extend the Lemma \ref{R4} for $k$-th term,  ${^L[G,H]}^{(k)} = {^L[{^L[G,H]}^{(k-1)},{^L[G,H]}^{(k-1)}]} $, of the derived series of ${^L[G,H]}$.
		
	\begin{lemma}\label{R10}
	Let $G$ and $H$ be multiplicative Lie algebras acting on each other compatibly.
	Suppose $x, y \in {^L[G,H]}^{(k)} $, then there exists $ z\in {^L[H,G]}^{(k)} $ such that  $x \star y = \langle z,y \rangle,$ where ${^L[G,H]}^{(k)}$  denote the $k$-th term of the derived series of ${^L[G,H]}$.
\end{lemma}	
\begin{proof}
To prove our claim, we proceed by induction on $k.$ For $k = 0,$ the claim reduces to Lemma \ref{R4}. Now suppose $k=1.$ Let  $x ={^L[x_1,x_2]}, \ y \in {^L[G,H]}^{(1)},$ where $x_1,x_2 \in {^L[G,H]}.$ Then by Lemma \ref{R4} and Proposition \ref{R2}, and taking into account the compatibility conditions of actions, we have   
	\begin{align*}
	{^L[x_1,x_2]}\star y &= {((x_1\star x_2)^{-1}[x_1,x_2])} \star y \\
	&= {(\langle y_1,x_2 \rangle^{-1}({^{x_1}{x_2}}{x_2}^{-1}))} \star y \\
	&= {(\langle y_1,x_2 \rangle^{-1}({^{y_1}{x_2}}{x_2}^{-1}))} \star y \\
	&= {^{\langle y_1,x_2 \rangle^{-1}}{({({^{y_1}{x_2}}{x_2}^{-1})} \star y )}} ({\langle y_1,x_2 \rangle^{-1} \star y }) \\
	&= {^{\langle x_2,y_1 \rangle}{\langle{{y_1}^{x_2}{y_1}^{-1}} , y \rangle}} \langle{\langle x_2,y_1 \rangle , y }\rangle \\
	&= \langle{\langle x_2,y_1 \rangle ({y_1}^{x_2}{y_1}^{-1}) , y }\rangle \\
	&= \langle{ (y_2\star y_1) ({y_1}^{y_2}{y_1}^{-1}) , y }\rangle \\
	&= \langle{{^L[y_1,y_2]} , y }\rangle\\
	&= \langle{z , y }\rangle, \ \text{where} \ z = {^L[y_1,y_2]} \in{^L[G,H]}^{(1)} .
\end{align*}
 Hence, for  $x, y\in {^L[G,H]}^{(1)} $, there exists $z \in {^L[H,G]} $ such that  $x \star y = \langle z, y \rangle.$ 
 
Now, for $k=2.$ Let $x = {^L[^L[x_1,x_2],^L[x_3,x_4]]}, \  y \in {^L[G,H]}^{(2)},$ where $x_1,x_2,x_3,x_4 \in {^L[G,H]}.$ Then by using Lemma \ref{R6}, we have  
\begin{equation*}
\begin{split}
^L[^L[x_1,x_2],^L[x_3,x_4]]\star y &= {((^L[x_1,x_2]\star {^L[x_3,x_4])^{-1}}[^L[x_1,x_2],^L[x_3,x_4]])} \star y\\
&\hspace{-2.5cm}= ^{(^L[x_1,x_2]\star ^L[x_3,x_4])^{-1}}({^{^L[x_1,x_2]}{^L[x_3,x_4]}}{^L[x_3,x_4]}^{-1} \star y) ((^L[x_1,x_2]\star {^L[x_3,x_4])^{-1}} \star y)\\
&\hspace{-2.5cm}= ^{\langle ^L[y_1,y_2],^L[x_3,x_4] \rangle^{-1}}({^{^L[y_1,y_2]}{^L[x_3,x_4]}}{^L[x_3,x_4]}^{-1} \star  y) (\langle ^L[y_1,y_2],^L[x_3,x_4] \rangle^{-1} \star y)\\
&\hspace{-2.5cm}	= ^{\langle ^L[x_3,x_4],^L[y_1,y_2] \rangle}(\langle{^L[y_1,y_2]}{^{^L[x_3,x_4]}}{^L[y_1,y_2]}^{-1},  y\rangle) \langle \langle ^L[x_3,x_4],^L[y_1,y_2] \rangle, y\rangle\\
&\hspace{-2.5cm}= \langle{\langle ^L[x_3,x_4],^L[y_1,y_2] \rangle}({^L[y_1,y_2]}{^{^L[x_3,x_4]}}{^L[y_1,y_2]}^{-1}), y\rangle\\
&\hspace{-2.5cm}= \langle{( ^L[y_3,y_4] \star {^L[y_1,y_2]} )}({^L[y_1,y_2]}{^{^L[y_3,y_4]}}{^L[y_1,y_2]}^{-1}), y\rangle \\
&\hspace{-2.5cm}	= \langle^L[^L[y_1,y_2],^L[y_3,y_4]], y\rangle\\
&\hspace{-2.5cm}= \langle z, y\rangle, \   \text{where} \ z = {^L[^L[y_1,y_2],^L[y_3,y_4]]} \in{^L[G,H]}^{(2)}.
\end{split}	
\end{equation*}
       
Now, we can observe that this result true for the general case, i.e, given $x, y \in {^L[G,H]}^{(k)} $, there exists $ z\in {^L[H,G]}^{(k)} $ such that  $x \star y = \langle z,y \rangle.$
\end{proof}

\begin{proposition} \label{R9}
Suppose the multiplicative Lie algebras $G$ and $H$ act compatibly on each other.
\begin{enumerate}
\item If $^L[G, H]$ is Lie nilpotent, then so is $^L[H, G]$ and $cl(^L[H, G])\leq cl(^L[G, H])+1.$\vspace{0.15cm}
\item  If $^L[G, H]$ is Lie solvable, then so is $^L[H, G]$ and $l(^L[H, G])\leq l(^L[G, H])+1.$
\end{enumerate}
\end{proposition}	
\begin{proof}$(1)$
Suppose $^L[G, H]$ is Lie nilpotent of class $n.$ Then $^L[z_1,z_2, \cdots, z_{n+1}] = 1_H$ for all $z_i \in {^L[G, H]}.$ To prove our claim, it suffices to show $^L[y_1,y_2, \cdots, y_{n+2}] = 1_G$ for all $y_i \in {^L[H, G]}. $ We have 
\begin{align*}
^L[y_1,y_2, \cdots, y_{n+2}] &=  (^L[y_1,y_2, \cdots, y_{n+1}]\star y_{n+2})^{-1} [^L[y_1,y_2, \cdots, y_{n+1}], y_{n+2} ]\\& = (^L[y_1,y_2, \cdots, y_{n+1}]\star y_{n+2})^{-1}{{^{^L[y_1,y_2, \cdots, y_{n+1}]}}{y_{n+2}}} {y_{n+2}}^{-1}.
\end{align*}
   Next, by Proposition \ref{R2}, there exist $x_1,x_2 \cdots x_{n+1} \in {^L[H, G]} $ such that each pair $x_i, y_i$ satisfies $^{x_i}g' =~ ^{y_i}g'$   and $ ^{x_i}h'  = {^{y_i}h'},$ for all $g'\in G $ and $h'\in H$. Thus, by repeated application of Lemma \ref{R6} and by the Lemma \ref{R10}, we have $^L[y_1,y_2, \cdots, y_{n+2}] = \big < ^L[x_1,x_2, \cdots, x_{n+1}], \ y_{n+2}  \big>^{-1} {{^{^L[x_1,x_2, \cdots, x_{n+1}]}}{y_{n+2}}} {y_{n+2}}^{-1}.$ Now, by assumption, we have $^L[x_1,x_2, \cdots, x_{n+1}] = 1,$ so  $^L[y_1,y_2, \cdots, y_{n+2}] = 1  $ for all  $y_i \in {^L[H, G]}.$ Hence, $^L[H, G]$ is Lie nilpotent with $cl(^L[H, G])\leq cl(^L[G, H])+1.$ 
  
$(2)$ For a group $K,$ let $K^{(k)}$ denote the $k$-th term of the derived series. We first claim that given $v \in {^L[H, G]^{(k)}}, $ there exists $u \in {^L[G, H]^{(k)}} $ so that the pair $u, v$ satisfies Proposition \ref{R2}. We proceed with the proof of the claim by induction on $k.$ For $k = 0,$ the claim reduces to Proposition \ref{R2}. Now suppose $k\geq1.$ Let $v_1, v_2 \in {^L[H, G]^{(k-1)}}.$ By the induction hypothesis, there exist  $u_1, u_2 \in {^L[H, h]^{(k-1)}}$ such that each of the pairs $u_i, v_i, i = 1,2$ satisfies Proposition \ref{R2}. So by Lemma \ref{R6}, the pair of Lie commutator $^L[u_1, u_2], ^L[v_1, v_2]$ also satisfies Proposition \ref{R2}. Since $^L[H, G]^{(k)}$ is generated by Lie commutator $^L[v_1, v_2],$ where $v_i \in  {^L[H, G]^{(k-1)}}, $ so the claim holds for the generators of $^L[H, G]^{(k)}.$ Now, by induction on the word length of elements of $^L[H, G]^{(k)}$, for given $v \in {^L[H, G]^{(k)}}, $ there exists $u \in {^L[G, H]^{(k)}} $ so that $^ug' = {^vg'}  \ \ \text{and} \ \ ^uh' = {^vh'}$ for all $g'\in G $ and  $h'\in H. $ Now suppose that $^L[G, H]$ is Lie solvable of derived length $n.$ Given a generators $^L[v, v'] \in {^L[H, G]^{(n+1)}},$ where $v, v'\in {^L[H, G]^{(n)}},$ by our claim there exists $u\in  {^L[G, H]^{(n)}} $ so that $^vv' = {^uv'}$ and by Lemma \ref{R10}, there exists $u'\in  {^L[G, H]^{(n)}} $ such that $v\star v' = \langle u', v' \rangle.$ However, by our hypothesis, $^L[G, H]^{(n)} = 1,$ so $u= u'=1,$ implying $^L[v, v'] =   (v\star v')^{-1} [v, v'] = \langle u', v' \rangle {^uv'} v'^{-1} = 1.$ We conclude that $^L[H, G]^{(n+1)} = 1,$ so $^L[H, G]$ is Lie solvable. Furthermore, we have $l(^L[H, G])\leq l(^L[G, H])+1 $ as claimed.      
\end{proof}

\begin{theorem}
Let $I$ and $J$ be ideals of a multiplicative Lie algebras $G$ and $H$, respectively, such that $I$ is invariant under the group action of $H,$ and $J$ is invariant under the group action of $G.$ Then the subgroup $I \otimes J = \big < \  (a\otimes b) \ | \ a \in I \ \text{and} \ b \in J  \  \big > $ is an ideal of multiplicative Lie algebra $(G\otimes H, ., \star).$
\end{theorem}

\begin{proof}
	Let $g\otimes h \in G\otimes H$ and $a\otimes b \in I\otimes J.$ Then by the Proposition \ref{s2} and \ref{s1}, we have $^{(g\otimes h)}(a\otimes b) = {^{[g, h]} (a\otimes b)} = (^{[g, h]}a\otimes {^{[g, h]} b}).$  Since $I$ and $J$ are ideals of $G$ and $H$ respectively, and  $I$ is invariant under the group action of $H,$ and $J$ is invariant under the group action of $G,$ so $^{(g\otimes h)}(a\otimes b) \in I\otimes J.$ Hence, any generator of $I\otimes J$ is mapped to another generator under the action of conjugation by elements of $G\otimes H.$ It follows $I\otimes J \trianglelefteq G\otimes H.$
	
	Next, by the identity $(5)$ (Definition \ref{s3}) and taking into account the compatibility of actions, we have 
	\begin{align*}
	\hspace{0.2cm}(g\otimes h) \star ( a\otimes b)
	 &=  {\langle h, g \rangle}^{-1} \otimes {\langle a, b \rangle}\\ 
	  &=  {\langle h, g \rangle}^{-1} \otimes {\langle ^ba, b^{-1} \rangle}^{-1}\\ 
	  &= {^{^ba}}{\langle {^{^ba^{-1}}h}, {^{^ba^{-1}}g} \rangle}^{-1} \otimes {{\langle ^ba, b^{-1} \rangle}^{-1}}\\
	  &= ((c \star {\langle ^{c^{-1}}h, {^{c^{-1}}g} \rangle}^{-1}) \otimes {^{ {\langle ^{c^{-1}}h, {^{c^{-1}}g} \rangle}^{-1}}b^{-1})}(^{b^{-1}}c \otimes \langle {\langle ^{c^{-1}}h, {^{c^{-1}}g} \rangle}^{-1}, b^{-1} \rangle  )^{-1} \\ &= ((c \star {\langle ^{c^{-1}}h, {^{c^{-1}}g} \rangle}^{-1}) \otimes {^{ {\langle ^{c^{-1}}g, {^{c^{-1}}h} \rangle}}b^{-1}})({^{b^{-1}}c} \otimes ({\langle ^{c^{-1}}g, {^{c^{-1}}h} \rangle}\star b^{-1} )  )^{-1} \\
	    &= ((c \star g') \otimes {^{ h'}b^{-1}})(^{b^{-1}}c \otimes  (h' \star b^{-1})   )^{-1}, 
	\end{align*}
	where, $^ba = c \in I, \ g' = {\langle {^{c^{-1}}h}, {^{c^{-1}}g} \rangle}^{-1} \in G \ \text{and} \ h' = {\langle ^{c^{-1}}g, {^{c^{-1}}h} \rangle} \in H  .$ Since, $I$ and $J$ are ideals and $I$ is invariant under the group action of $H,$ $(g\otimes h) \star ( a\otimes b) \in I\otimes J.$  Hence, $I \otimes J$ is an ideal of $G\otimes H.$
\end{proof}

\begin{lemma}\label{Dee}
Let $G$ and $H$ be multiplicative Lie algebras acting on each other in a compatible way. Then for $g, g'\in G $ and $h, h' \in H,$ element $^L[(g\otimes h), (g^{\prime}\otimes h^{\prime})]$ can be expressed as $$^{{^L[h,g]^{-1}}{\langle g', h' \rangle}}({\langle h, g \rangle}^{-1}   \otimes  {^L[g',h']})   ({^L[h,g]^{-1}}   \otimes {\langle g', h' \rangle}) ({^L[h,g]^{-1}}   \otimes {^L[g',h']}).$$

\end{lemma}
\begin{proof}
	Suppose $g, g'\in G $ and $h, h' \in H,$ then by identity $(5)$ (Definition \ref{s3}) and Proposition \ref{s1}, we have
	\begin{equation*}
	\begin{split}
		\hspace{0.5cm} ^L[(g\otimes h), (g^{\prime}\otimes h^{\prime})] & = ((g\otimes h) \star (g^{\prime}\otimes h^{\prime}))^{-1}[  g\otimes h, g^{\prime}\otimes h^{\prime}]\\
		& = ( \langle h, g \rangle ^{-1} \otimes \langle g', h' \rangle)^{-1} (g^h{g^{-1}} \otimes {^{g'}{h'}{h'}^{-1})}\\
		& = ( \langle h, g \rangle ^{-1} \otimes \langle g', h' \rangle)^{-1} ((^h{g}g^{-1})^{-1} \otimes {^{g'}{h'}{h'}^{-1})}\hspace{3cm} (1) 
	\end{split}
	\end{equation*}
	Now, by Proposition \ref{p1} and taking into account the compatibility conditions of actions, we have
	\begin{equation*}
		\begin{split}
			\hspace{-0.01cm} (^h{g}g^{-1})^{-1} \otimes {^{g'}{h'}{h'}^{-1}} &= (^h{g}g^{-1})^{-1} {\langle h, g \rangle}  {\langle h, g \rangle}^{-1}   \otimes {\langle g', h' \rangle} {\langle g', h' \rangle}^{-1} {^{g'}{h'}{h'}^{-1}} \\
			&= {(^L[h,g])^{-1}} {\langle h, g \rangle}^{-1}   \otimes {\langle g', h' \rangle} ^L[g',h'] \\
			&= ^{{(^L[h,g])^{-1}}}({\langle h, g \rangle}^{-1}   \otimes {\langle g', h' \rangle} ^L[g',h']) ({(^L[h,g])^{-1}}   \otimes {\langle g', h' \rangle} ^L[g',h']) \\
			&= ^{{(^L[h,g])^{-1}}}({\langle h, g \rangle}^{-1}   \otimes {\langle g', h' \rangle}) ^{{(^L[h,g])^{-1}}{\langle g', h' \rangle}}({\langle h, g \rangle}^{-1}   \otimes  {^L[g',h']}) \\& \hspace{0.5cm} ({(^L[h,g])^{-1}}   \otimes {\langle g', h' \rangle}) ^{{\langle g', h' \rangle}}({(^L[h,g])^{-1}}   \otimes {^L[g',h']}) \\
			&= ({\langle h, g \rangle}^{-1}   \otimes {\langle g', h' \rangle}) ^{{(^L[h,g])^{-1}}{\langle g', h' \rangle}}({\langle h, g \rangle}^{-1}   \otimes  {^L[g',h']}) \\& \hspace{0.5cm}  ({(^L[h,g])^{-1}}   \otimes {\langle g', h' \rangle})  ({(^L[h,g])^{-1}}   \otimes {^L[g',h']}) 
		\end{split}
	\end{equation*}
	Put the value of $(^h{g}g^{-1})^{-1} \otimes {^{g'}{h'}{h'}^{-1}}$ in equation (1), we get desired result
	\begin{equation*}
		\begin{split}
			\hspace{-.1cm} ^L[(g\otimes h), (g^{\prime}\otimes h^{\prime})] & = {^{{{^L[h,g]^{-1}}}{\langle g', h' \rangle}}}({\langle h, g \rangle}^{-1}   \otimes  {^L[g',h']})   ({^L[h,g]^{-1}}   \otimes {\langle g', h' \rangle}) \\& \hspace{0.5cm}  ({^L[h,g]^{-1}}   \otimes {^L[g',h']}).
		\end{split}
	\end{equation*}
\end{proof}

We can now establish upper bounds on the Lie nilpotency class and Lie solvability length of $ \frac{G \otimes H}{I},$ for some ideal $I$ of $G \otimes H,$ in terms of the Lie nilpotency class and Lie solvability length of $H.$ To begin, we provide upper bounds on the Lie nilpotency class and Lie solvability length of  $ \frac{G\otimes H}{I},$ for some ideal $I$ of $G \otimes H.$ 

\begin{theorem}\label{nilotent}
	Let $G$ and $H$ be multiplicative Lie algebras act compatibly on each other. If $ H$ is Lie nilpotent  of class $n,$ then $ \frac{G\otimes H}{^L[H, G] \otimes \langle G, H \rangle }$ is also Lie nilpotent of class $ \leq n+1 $ .  
	\end{theorem}
\begin{proof}
	To prove our theorem, it suffices to show $L_{n + 1}( \frac{G\otimes H}{K}) \subseteq   \frac{(({\langle H, G \rangle\otimes L_{n}( H))(^L[H, G]\otimes L_{n}(H))})K}{K},$ where $K = {^L[H, G] \otimes \langle G, H \rangle }.$  We proceed the proof by induction on $n.$ Assume $n = 0,$
	
Let $ ^L[t_1K, t_2K] \in{^L[\frac{G\otimes H}{K}, \frac{G\otimes H}{K} ]} = L_{1}(\frac{G\otimes H}{K})$, where $t_1, t_2 \in  G\otimes H$. Suppose $t_1 = g\otimes h$ and $t_2 = g^{\prime}\otimes h^{\prime}$. Then by Lemma \ref{Dee}, Proposition \ref{p1}, and Remark \ref{d1}, we have
 \begin{equation*}
		\begin{split}
			^L[(g\otimes h)K, (g^{\prime}\otimes h^{\prime})K] & = {^L[(g\otimes h), (g^{\prime}\otimes h^{\prime})]K} \\
			& \hspace{-3cm}=   {^{{{(^L[h,g])^{-1}}}{\langle g', h' \rangle}}}({\langle h, g \rangle}^{-1}   \otimes  {^L[g',h']})   ({(^L[h,g])^{-1}}   \otimes {\langle g', h' \rangle}) ({(^L[h,g])^{-1}}   \otimes {^L[g',h']})K\\ 
			&\hspace{-3cm}= {^{{{(^L[h,g])^{-1}}}{\langle g', h' \rangle}}}({\langle h, g \rangle}^{-1}   \otimes  {^L[g',h']})   ^{({(^L[h,g])^{-1}}   \otimes {\langle g', h' \rangle})}({(^L[h,g])^{-1}}   \otimes {^L[g',h']}) \\& \hspace{-2.4cm} ({(^L[h,g])^{-1}}   \otimes {\langle g', h' \rangle})K\\ 
			&\hspace{-3cm}= {^{{{(^L[h,g])^{-1}}}{\langle g', h' \rangle}}}({\langle h, g \rangle}^{-1}   \otimes  {^L[g',h']})   ^{[{(^L[h,g])^{-1}}, {\langle g', h' \rangle}]}({(^L[h,g])^{-1}}   \otimes {^L[g',h']})K\\ 
			&\hspace{-3cm}= {^{{{(^L[h,g])^{-1}}}{\langle g', h' \rangle}}}({\langle h, g \rangle}^{-1}   \otimes  {^L[g',h']})({(^L[h,g])^{-1}}   \otimes {^L[g',h']})K \hspace{3.4cm}{(1)}
			\end{split}
	\end{equation*}
	This shows that $^L[(g\otimes h)K, (g^{\prime}\otimes h^{\prime})K]  \in   \frac{({\langle H, G \rangle\otimes L_0(H)})(^L[H, G]\otimes L_0(H))K}{K}.$
	
	 Since $L_{1}(\frac{G\otimes H}{K})$ is generated by the elements of the type $^L[(g\otimes h)K, (g^{\prime}\otimes h^{\prime})K]$, $ ^L[t_1K, t_2K]\in \frac{({\langle H, G \rangle\otimes L_0(H)})(^L[H, G]\otimes L_0(H))K}{K}.$ Thus $L_{1}( \frac{G\otimes H}{K}) \subseteq  \frac{({\langle H, G \rangle\otimes L_0(H)})(^L[H, G]\otimes L_0(H))K}{K}.$

	Now suppose $L_{n}( \frac{G\otimes H}{K}) \subseteq   \frac{(({\langle H, G \rangle\otimes L_{n-1}( H))(^L[H, G]\otimes L_{n-1}(H))})K}{K}.$ Let $P \in {^L[\frac{G\otimes H}{K}, L_{n}(\frac{G\otimes H}{K}) ]} = L_{n + 1}(\frac{G\otimes H}{K}). $ By the induction hypothesis, $P = {^L[\alpha K,  ({\langle h', g' \rangle} \otimes   L_{n-1}(h_2))( {^L[h,g]} \otimes  L_{n-1}(h_1))K]},$ for some $\alpha \in G\otimes H$ and $L_{n-1}(h_2), L_{n-1}(h_2) \in L_{n-1}(H) $. Suppose $\alpha = x\otimes y$. Then by Proposition \ref{Lie}, we have	
	\begin{equation*}
		\begin{split}
		\hspace{1cm}P & = {^L[(x\otimes y)K,  ({\langle h', g' \rangle} \otimes  L_{n-1}(h_2))({^L[h,g]} \otimes  L_{n-1}(h_1))K]}\\   
		&= {^{b}{(^L[a, c])} \ ^{[^{b}c, ^{b}a]}({^L[a, b]})}K, \hspace{9cm} (2)
		\end{split}
	\end{equation*}

 where $ a = x\otimes y, \ b = {\langle h', g' \rangle} \otimes  L_{n-1}(h_2), \text{ and}\  c ={^L[h,g]} \otimes  L_{n-1}(h_1).$\\
 Now, by equation (1), we have  
	\begin{equation*}
		\begin{split}
			^L[a, b] &= {^L[(x\otimes y), ({\langle h', g' \rangle} \otimes  L_{n-1}(h_2))]} \\
			&= {^{{{(^L[y,x])^{-1}}}{\langle {\langle h', g' \rangle}, L_{n-1}(h_2) \rangle}}}({\langle y, x \rangle}^{-1}   \otimes  {^L[{\langle h', g' \rangle},L_{n-1}(h_2)]})\\& \hspace{0.6cm}({(^L[y,x])^{-1}}   \otimes {^L[{\langle h', g' \rangle},L_{n-1}(h_2)]}) \\
			&= {^{t}}({\langle y, x \rangle}^{-1}   \otimes {\langle {\langle h', g' \rangle}, L_{n-1}(h_2) \rangle}^{-1}  {^{\langle h', g' \rangle}L_{n-1}(h_2)}{L_{n-1}(h_2)^{-1}})\\& \hspace{0.6cm}({(^L[y,x])^{-1}}   \otimes {\langle {\langle h', g' \rangle}, L_{n-1}(h_2) \rangle}^{-1}  {^{\langle h', g' \rangle}L_{n-1}(h_2)}{L_{n-1}(h_2)^{-1}}) \\
		    &= {^{t}}({\langle y, x \rangle}^{-1}   \otimes ({ {\langle g', h' \rangle}^{-1}\star L_{n-1}(h_2))}^{-1}  {^{\langle g', h' \rangle^{-1}}L_{n-1}(h_2)}{L_{n-1}(h_2)^{-1}})\\& \hspace{0.6cm}({(^L[y,x])^{-1}}   \otimes({\langle {\langle g', h' \rangle}^{-1}\star L_{n-1}(h_2))}^{-1}  {^{\langle g', h' \rangle^{-1}}L_{n-1}(h_2)}{L_{n-1}(h_2)^{-1}}) \\
		     &= {^{t}}({\langle y, x \rangle}^{-1}   \otimes {^L[{\langle g', h' \rangle}^{-1}, L_{n-1}(h_2)}]  ) ({(^L[y,x])^{-1}}   \otimes {^L[{\langle g', h' \rangle}^{-1}, L_{n-1}(h_2)]}),
		\end{split}
	\end{equation*}
	where $t= {{(^L[y,x])^{-1}}}{\langle {\langle h', g' \rangle}, L_{n-1}(h_2) \rangle}.$ Similarly by equation $(1)$, we can calculate $^L[a, b]$ by using compatibility conditions. Thus
	\begin{equation*}
		\begin{split}
			\hspace{-.2cm} ^L[a, c] &= {^L[(x\otimes y), ({^L[h,g]} \otimes  L_{n-1}(h_1))]} \\
				&= {^{s}}({\langle y, x \rangle}^{-1}   \otimes  {^L[{\langle g, h \rangle}h{^gh^{-1}},L_{n-1}(h_1)]})({(^L[y,x])^{-1}}   \otimes {^L[{\langle g, h \rangle}h{^gh^{-1}},L_{n-1}(h_1)]}),
		\end{split}
	\end{equation*}
where $s = {{(^L[y,x])^{-1}}}{\langle {^L[h,g]}, L_{n-1}(h_1) \rangle}.$ Now put the value of $^L[a, b]$ and $^L[a, c]$ in equation (2), we get
	\begin{equation*}
		\begin{split}
		    P & =  {^{b}} ({^{s}}({\langle y, x \rangle}^{-1}   \otimes  {^L[{\langle g, h \rangle}h{^gh^{-1}},L_{n-1}(h_1)]})({(^L[y,x])^{-1}}   \otimes {^L[{\langle g, h \rangle}h{^gh^{-1}},L_{n-1}(h_1)]}))\\ & \hspace{0.8cm} {^{[^{b}c, ^{b}a]}} ({^{t}}({\langle y, x \rangle}^{-1}   \otimes {^L[{\langle g', h' \rangle}^{-1}, L_{n-1}(h_2)}]  ) ({(^L[y,x])^{-1}}   \otimes {^L[{\langle g', h' \rangle}^{-1}, L_{n-1}(h_2)]})) K \\
		    & =  {^{bs}}({\langle y, x \rangle}^{-1}   \otimes  {^L[{\langle g, h \rangle}h{^gh^{-1}},L_{n-1}(h_1)]}){^{{^{b}}({(^L[y,x])^{-1}}   \otimes {^L[{\langle g, h \rangle}h{^gh^{-1}},L_{n-1}(h_1)]})[^{b}c, ^{b}a]}} ({^{t}}({\langle y, x \rangle}^{-1}   \otimes \\ & \hspace{0.8cm}{^L[{\langle g', h' \rangle}^{-1}, L_{n-1}(h_2)}]  ) ({(^L[y,x])^{-1}}   \otimes {^L[{\langle g', h' \rangle}^{-1}, L_{n-1}(h_2)]}))  {^{b}}({(^L[y,x])^{-1}}   \otimes \\& \hspace{0.8cm} {^L[{\langle g, h \rangle}h{^gh^{-1}},L_{n-1}(h_1)]})  K \\
		    & =  {^{bs}}({\langle y, x \rangle}^{-1}   \otimes  {^L[{\langle g, h \rangle}h{^gh^{-1}},L_{n-1}(h_1)]}){^{{^{b}}({[^L[y,x])^{-1}}, {^L[{\langle g, h \rangle}h{^gh^{-1}},L_{n-1}(h_1)]}] [^{b}c, ^{b}a]}} ({^{t}}({\langle y, x \rangle}^{-1}   \otimes \\ & \hspace{0.8cm}{^L[{\langle g', h' \rangle}^{-1}, L_{n-1}(h_2)}]  ) ({(^L[y,x])^{-1}}   \otimes {^L[{\langle g', h' \rangle}^{-1}, L_{n-1}(h_2)]}))  {^{b}}({(^L[y,x])^{-1}}   \otimes \\& \hspace{0.8cm} {^L[{\langle g, h \rangle}h{^gh^{-1}},L_{n-1}(h_1)]})  K \\
		     & =  {^{bs}}({\langle y, x \rangle}^{-1}   \otimes  {^L[{\langle g, h \rangle}h{^gh^{-1}},L_{n-1}(h_1)]}){^{z}} ({^{t}}({\langle y, x \rangle}^{-1}   \otimes {^L[{\langle g', h' \rangle}^{-1}, L_{n-1}(h_2)}]  )\\ & \hspace{0.8cm} ({(^L[y,x])^{-1}}   \otimes {^L[{\langle g', h' \rangle}^{-1}, L_{n-1}(h_2)]}))  {^{b}}({(^L[y,x])^{-1}}   \otimes {^L[{\langle g, h \rangle}h{^gh^{-1}},L_{n-1}(h_1)]})  K \\ &\in \frac{({\langle H, G \rangle\otimes L_{n}( H)})(^L[H, G]\otimes L_{n}(H))K}{K},
		\end{split}
	\end{equation*}
	where, $z = {^b}({[^L[y,x])^{-1}}, {^L[{\langle g, h \rangle}h{^gh^{-1}},L_{n-1}(h_1)]}]) [^{b}c, ^{b}a].$ Similarly, we can see that $P = {^L[\alpha K,  ({\langle h', g' \rangle} \otimes   L_{n-1}(h_2))( {^L[h,g]} \otimes  L_{n-1}(h_1))K]}\in \frac{({\langle H, G \rangle\otimes L_{n}( H)})(^L[H, G]\otimes L_{n}(H))K}{K}.$ Therefore, $L_{n + 1}( \frac{G\otimes H}{K}) \subseteq   \frac{(({\langle H, G \rangle\otimes L_{n}( H))(^L[H, G]\otimes L_{n}(H))})K}{K},$ for every $n \in \mathbb{N}.$
	
If $ H$ is Lie nilpotent  of class $n,$	then $L_n (H)= 1,$. So $L_{n + 1}( \frac{G\otimes H}{K}) = 1.$ Hence, $cl( \frac{G\otimes H}{   ^L[H, G] \otimes \langle G, H \rangle }) \leq n+1. $   
\end{proof}

\begin{remark}\begin{enumerate}
\item Like the proof of Theorem \ref{nilotent}, we can also prove that if $G$ and $H$ are multiplicative Lie algebras act compatibly on each other and $ H$ is Lie solvable  of class $n,$ then $ \frac{G\otimes H}{^L[H, G] \otimes \langle G, H \rangle }$ is also Lie solvable of class $ \leq n+1 $.  
\item Suppose $G=H$ and $\langle, \rangle = \star$. Then $\frac{G\otimes G}{^L[G, G] \otimes (G\star G)}$ is Lie nilpotent (Lie solvable) provided $G$ is Lie nilpotent (Lie solvable).
\item Since every element of the ideal $(G\star G) \otimes ^L[G, G]$ can be written as product of the elements of the ideal $^L[G, G] \otimes (G\star G)$, by Lemma \ref{Dee}, it is easy to see that $L_{n}( \frac{G\otimes G}{K}) \subseteq   \frac{(G\otimes L_{n}(G))K}{K},$ where $K= {^L[G, G] \otimes {^L[G, G]}}.$ Hence, if $G$  is a Lie nilpotent of class $n$, then $ \frac{G\otimes G}{^L[G, G] \otimes {^L[G, G]}}$  is also Lie nilpotent of class $n$.  Similarly, if $G$  is a Lie solvable, then $ \frac{G\otimes G}{^L[G, G] \otimes {^L[G, G]}}$  is also Lie solvable. 
\end{enumerate}
 
\end{remark} 

\noindent{\bf Acknowledgement:}
The first named author sincerely thanks IIIT Allahabad and Ministry of Education, Government of India for providing institute fellowship. The second named author sincerely thanks IIIT Allahabad and University grant commission (UGC), Govt. of India, New Delhi for research fellowship.

\end{document}